\documentclass[ASNA]{USG} 
\usepackage{anyfontsize} %

\usepackage{colortbl}   
\usepackage{xcolor}     

\usepackage{enumitem}
\usepackage{tikz}
\usetikzlibrary{calc}

\usepackage{steinmetz}

\articletype{MANUSCRIPT}%

\journal{}

\startpage{1}
\articledoi{}

\def\dvt{d_{vt}}
\def\Sym{{\rm Sym}}
\def\Aut{{\rm Aut}}
\newcommand{\Y}{\mathrm{Y}}
\newcommand{\ZZ}{\mathbb Z}
\newcommand{\D}{\mathrm{D}}
\def\Cay{{\rm Cay}}
\def\gvt{g_{vt}}

\begin{document}
\title{Vertex-transitive closures of graphs}
\author[1]{Martin Bachrat{\' y}}[https://orcid.org/0000-0002-4300-7507]
\author[2]{Grahame Erskine}[https://orcid.org/0000-0001-7067-6004]
\author[3]{{\v S}tef{\' a}nia Glevitzk{\' a}}[https://orcid.org/0009-0003-4810-0326]
\author[1,2]{Jozef {\v S}ir{\' a}{\v n}}[https://orcid.org/0000-0002-5901-7646]
\author[2,4]{James Tuite}[https://orcid.org/0000-0003-2604-7491]

\address[1]{\orgdiv{Department of Mathematics and Descriptive Geometry, }\orgname{Slovak University of Technology, }%
\orgaddress{\state{Bratislava, }\country{Slovakia}}}

\address[2]{\orgdiv{School of Mathematics and Statistics, }\orgname{Open University, }%
\orgaddress{\state{Milton Keynes, }\country{UK}}}

\address[3]{\orgdiv{Department of Algebra and Geometry, }\orgname{Comenius University, }%
\orgaddress{\state{Bratislava, }\country{Slovakia}}}

\address[4]{\orgdiv{Department of Informatics and Statistics, }\orgname{Klaip\.{e}da University, }%
\orgaddress{\country{Lithuania}}}

\corres{James Tuite  (\email{james.t.tuite@open.ac.uk})}


\abstract[ABSTRACT]{A \emph{vertex-transitive closure} of $\Gamma$ is a vertex-transitive supergraph of $\Gamma$ on the same vertex set. The \emph{vertex-transitive number} of a graph $\Gamma$, denoted by $\dvt(\Gamma)$, is the smallest integer for which there exists a $\dvt(\Gamma)$-regular vertex-transitive closure of $\Gamma$. In this paper we use various algebraic and combinatorial methods to study vertex-transitive closures of graphs and the associated vertex-transitive number.}


\maketitle


	\section{Introduction}
	
	Let $\Gamma$ be a finite simple graph. A \emph{vertex-transitive closure} of $\Gamma$ is a vertex-transitive supergraph of $\Gamma$ on the same vertex set. The \emph{vertex-transitive number} of a graph $\Gamma$, denoted by $\dvt(\Gamma)$, is the smallest integer for which there exists a $\dvt(\Gamma)$-regular vertex-transitive closure of $\Gamma$. The concept of a vertex-transitive closure was first briefly introduced in~\cite{BachratySiran}, where it was proved that for every odd prime power $q\geq 37$ there is no vertex-transitive closure of the polarity graph $B(q)$ of degree $q+3$. 
	
	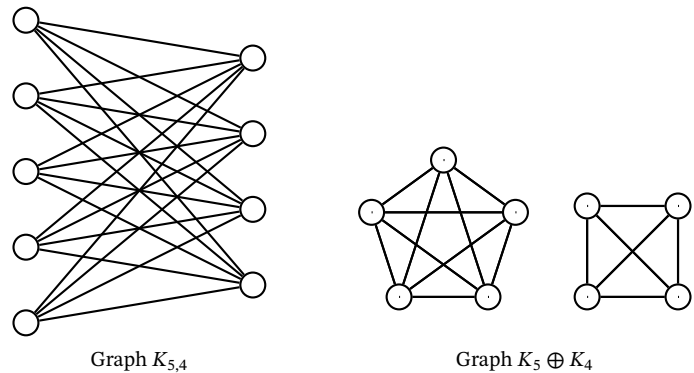
\begin{figure}[htb]
\centering
\subfloat[Graph $K_{5,4}$\label{fig:bigvt:a}]{
\begin{tikzpicture}[thick]
\foreach \i in {1,2,...,5}
\node[circle, draw] (A\i) at (0,\i) {};

\foreach \j in {1,2,...,4}
\node[circle, draw] (B\j) at (3,0.5+\j) {};

\foreach \i in {1,2,...,5}
\foreach \j in {1,2,...,4}
\draw (A\i)--(B\j);
\end{tikzpicture}
}
\quad\quad\quad
\subfloat[Graph $K_5 \oplus K_4$\label{fig:bigvt:b}]{
\begin{tikzpicture}[thick]
\foreach \i in {1,2,...,5}
\node[circle, draw] (A\i) at (18+\i*72:1) {};

\foreach \i in {1,2,...,5}
\foreach \j in {1,2,...,5}
\draw (A\i)--(A\j);

\foreach \i in {1,2,...,4}
\node[circle, draw] (B\i) at ($(45+\i*90:0.85)+(2.5,-0.208)$) {};

\foreach \i in {1,2,...,4}
\foreach \j in {1,2,...,4}
\draw (B\i)--(B\j);

\node[draw=white] at (0,-1.2) {};
\end{tikzpicture}
}
\caption{The complete bipartite graph $K_{5,4}$ and its complement $K_5 \oplus K_4$}
\label{fig:bigvt}
\end{figure}	
	
	\begin{example}\label{ex:bigvt}
Let $K_{5,4}$ be the complete bipartite graph shown in Figure~\ref{fig:bigvt:a}. The automorphism group of $K_{5,4}$ is isomorphic to the direct product of the symmetric groups $\Sym(5)$ and $\Sym(4)$, and so the order of $\Aut(K_{5,4})$ is $2880$. Hence $|\Aut(K_{5,4})|$ is much larger than $|V(K_{5,4})|$, and in this sense we may view $K_{5,4}$ as a highly symmetric graph. Next we find the vertex-transitive number of $K_{5,4}$. First, note that if a graph is vertex-transitive, so is its complement. Hence every vertex-transitive closure of $K_{5,4}$ corresponds to a vertex-transitive spanning subgraph $\Gamma$ of the complement of $K_{5,4}$, which is the disjoint union of $K_5$ and $K_4$ displayed in Figure~\ref{fig:bigvt:b}. Note that all components of $\Gamma$ have the same number of vertices, and this number must divide both $5$ and $4$. It follows that $\Gamma$ must be the edgeless graph on $9$ vertices, and consequently the only vertex-transitive closure of $K_{5,4}$ is the complete graph $K_9$. Hence we have $\dvt(K_{5,4})=8$. So although $K_{5,4}$ has a rich group of automorphisms, it is far from being a vertex-transitive graph in the sense that the difference between the maximum degree of $K_{5,4}$ and the vertex-transitive number of $K_{5,4}$ is relatively large.
\end{example}

	More generally, if $r_1 \geq r_2 \geq \dots \geq r_t \geq 1$ for $t \geq 2$, then a vertex-transitive supergraph of the complete multipartite graph $K_{r_1,r_2,\dots ,r_t}$ corresponds to a vertex-transitive spanning subgraph $\Gamma $ of $K_{r_1} \oplus K_{r_2} \oplus \dots \oplus K_{r_t}$. The components of $\Gamma $ must have order dividing $\gcd (r_1,r_2,\dots ,r_t)$. Considering any vertex of $K_{r_1,r_2,\dots ,r_t}$ in the part of size $r_1$, it follows that the degree of $K_{r_1,r_2,\dots ,r_t}$ must be increased by at least $r_1-\gcd (r_1,r_2,\dots ,r_t)-1$ to achieve a vertex-transitive graph and this can be accomplished by setting $\Gamma $ to be the union of $\frac{r_1+r_2+\dots+r_t}{\gcd (r_1,r_2,\dots ,r_t)}$ cliques.

	\begin{example}\label{ex:asy}
		Let $\Gamma$ be the asymmetric graph shown in Figure~\ref{fig:smallvt:b}. Since $\Gamma$ is asymmetric, we know that $|\Aut(\Gamma)|=1$. On the other hand, $\Gamma$ is a spanning subgraph of the triangular prism graph $\Y_3$ shown in Figure~\ref{fig:smallvt:a}, and since $\Y_3$ is vertex-transitive, we find that $\dvt(\Gamma)=\Delta(\Gamma)$. Since there are no asymmetric graphs on five or fewer vertices~\cite{ErdosRenyi}, this is the smallest example (in terms of $|V(\Gamma)|$) of an asymmetric graph with its vertex-transitive number equal to its maximum degree. The smallest example in terms of $|E(\Gamma)|$ is the graph $\Gamma'$ obtained from $\Gamma$ by removing the bottommost edge. (To see this note that the triangular prism graph is a vertex-transitive closure of $\Gamma'$.)
	\end{example}

\begin{figure}[htb]
\centering
\subfloat[Triangular prism graph\label{fig:smallvt:a}]{
\begin{tikzpicture}[thick]
\foreach \i in {1,2,3}
\node[circle, draw] (A\i) at (120*\i+90:0.8) {};

\foreach \i in {1,2,3}
\node[circle, draw] (B\i) at (120*\i+90:2) {};

\foreach \i in {1,2,3}
\draw (A\i)--(B\i);

\draw (A1) -- (A2) -- (A3) -- (A1);
\draw (B1) -- (B2) -- (B3) -- (B1);
\end{tikzpicture}
}
\quad\quad\quad
\subfloat[An asymmetric graph on $6$ vertices\label{fig:smallvt:b}]{
\begin{tikzpicture}[thick]
\foreach \i in {1,2,3}
\node[circle, draw] (A\i) at (120*\i+90:0.8) {};

\foreach \i in {1,2,3}
\node[circle, draw] (B\i) at (120*\i+90:2) {};

\draw (A1) -- (A2) -- (A3) -- (A1);

\draw (A1) -- (B1) -- (B2);
\draw (A2) -- (B2) -- (B3);
\end{tikzpicture}
}
\caption{Triangular prism graph and its asymmetric spanning subgraph}
\label{fig:smallvt}
\end{figure}
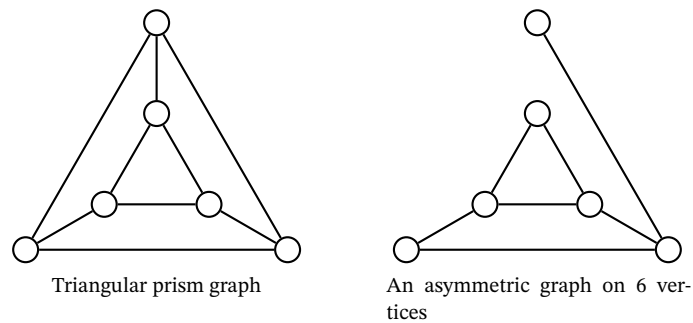

	\section{Further background}
	
	All graphs in this paper are assumed to be simple and, unless stated otherwise, finite. The vertex set and the edge set of a graph $\Gamma$ will be denoted by $V(\Gamma)$ and $E(\Gamma)$ respectively, and the sizes of these sets will be referred to as the \emph{order} and the \emph{size} of $\Gamma$. We let $K_n$, $P_n$, and $S_n$ denote the complete graph, the path graph, and the star graph of order $n$, and $K_{m,n}$ the complete bipartite graph (of order $m+n$) with parts of size $m$ and $n$. The complement of a graph $\Gamma$, the disjoint union of graphs $\Gamma_1$ and $\Gamma_2$, and their join will be denoted by $\overline{\Gamma}$, $\Gamma_1 \oplus \Gamma_2$, and $\Gamma_1 \vee \Gamma_2$, respectively. We also let $\ZZ_n$ stand for the cyclic group of order $n$ and $\D_n$ the dihedral group of order $2n$. 
	
	For a graph $\Gamma$ and a vertex $u$ of $\Gamma $, we write $\deg_{\Gamma}(u)$ for the degree of $u$, and $\Delta(\Gamma)$ for the maximum degree of $\Gamma$. The only graphs with maximum degree $0$ are edgeless graphs. Every graph with maximum degree $1$ is the disjoint union of at least one $K_2$ and arbitrarily many (possibly zero) isolated vertices. Graphs with maximum degree $2$ are disjoint unions of cycles, paths and isolated vertices with at least one component being either a cycle or a path of length at least $2$.
	
	For a symmetric subset $X$ of a finite group $G$ such that $1_G\notin X$, we let $\Cay(G,X)$ signify the corresponding Cayley graph. Note that $\Cay(G,X)$ is connected if and only if $\langle X \rangle =G$. We let $M_n$ and $Y_n$ denote the M\"{o}bius ladder and the prism graph, respectively. It is well known that both $M_n$ and $Y_n$ are Cayley graphs with $M_n = \Cay(\ZZ_{2n},\{1,-1,n\})$ and $Y_n = \Cay(\D_n,\{r,r^{-1},s\})$ (here $r$ and $s$ denote the standard generators of $\D_n$).
	
	The following well-known theorem of Dirac from~\cite{Dirac} will be useful. 
	
	\begin{theorem}[Dirac's theorem]
		\label{thm:dirac}
		Let $\Gamma$ be a graph of order $n$, with $n\geq 3$. If the minimum degree of $\Gamma$ is at least $n/2$, then $\Gamma$ is Hamiltonian.
	\end{theorem}
	
	We will also use the following observation about matchings in connected vertex-transitive graphs. 
	
	\begin{theorem}[\cite{GodsilRoyle}]
		\label{thm:gr}
		Every connected vertex-transitive graph has a matching that misses at most one vertex.
	\end{theorem}
	
	\section{Graphs with a unique vertex-transitive closure}
	
	Noting that every graph $\Gamma$ of order $n$ has a vertex-transitive closure $K_n$, it follows that $\Gamma$ has a unique vertex-transitive closure if and only if $\dvt(\Gamma)=n-1$. In this section we answer various questions about graphs that have a unique vertex-transitive closure. First, we address the problem of determining the minimum possible number of edges of such graphs.

	\begin{theorem}\label{thm:edges}
		A graph of order $n$ with a unique vertex-transitive closure has at least $n-2$ edges for $n$ odd and at least $n-1$ edges for $n$ even. Moreover, the only graphs of odd order with a unique vertex-transitive closure and $n-2$ edges are $K_3 \oplus 2K_1$ and $S_{n-1} \oplus K_1$, and the only graphs of even order with a unique vertex-transitive closure and $n-1$ edges are $K_3 \oplus K_1$ and $S_n$.
	\end{theorem}
	\begin{proof}
		Since the complement of a Hamilton cycle is vertex transitive, if $\overline {\Gamma }$ is Hamiltonian, then it follows that $\Gamma $ has $\dvt(\Gamma )$ at most $n-3$. It is shown in~\cite[Theorem~4.3]{Ore2} that the largest possible size of a non-Hamiltonian graph is ${n-1 \choose 2}+1$ and any such non-Hamiltonian graph is isomorphic either to a clique $K_{n-1}$ with a leaf attached to one vertex, or to $K_2 \vee \overline{K_3}$. Suppose that $\Gamma $ has $m \geq n-2$ edges; then $\overline {\Gamma}$ has size at least ${n-1 \choose 2}+1$, so, by the previous observation, any graph with a unique vertex-transitive closure has at least $n-2$ edges. Moreover, any graph with a unique vertex-transitive closure and size exactly $n-2$ must be either $K_3 \oplus 2K_1$ or $S_{n-1} \oplus K_1$. 
		
		It is easily verified that $K_3 \oplus 2K_1$ has a unique vertex-transitive completion. Observe that any non-complete vertex-transitive completion of $S_{n-1} \oplus K_1$ must be $(n-2)$-regular. This is impossible for odd $n$, so $S_{n-1} \oplus K_1$ does have a unique vertex-transitive completion if $n$ is odd. 
		
		For a graph $\Gamma $ with even order $n$ and a unique vertex-transitive closure, $\overline {\Gamma }$ cannot contain a perfect matching, for otherwise $\Gamma $ could be completed to a clique minus a perfect matching. As the complement of $S_{n-1} \oplus K_1$ does have a perfect matching, no graph with even order and size $n-2$ has a unique vertex-transitive closure. Next, suppose that $\Gamma$ is a graph of even order $n$ with a unique vertex-transitive closure and $n-1$ edges. As  $\overline{\Gamma}$ has no perfect matching, it also has no Hamiltonian path. According to~\cite[Theorem~4.1]{Ore2}, any graph of size ${n-1 \choose 2}$ without a Hamiltonian path is isomorphic either to $K_{n-1} \oplus K_1$ or $S_4$. It can be easily seen that the complements of these graphs (isomorphic to $S_n$ and $K_3\oplus K_1$) do have a unique vertex-transitive completion.     
	\end{proof}
	
	In Theorem~\ref{thm:edges} we looked at the minimum possible number of edges in a graph with a unique vertex-transitive closure. In the following theorem we classify all possible maximal degrees of such graphs.  
	
	\begin{theorem}\label{thm:degrees}
		Let $n\geq 4$, and let $S$ be the set of all graphs of order $n$ with a unique vertex-transitive closure. Then the maximal degrees of graphs in $S$ are exactly the integers from $\lfloor n/2 \rfloor$ to $n-1$. 
	\end{theorem}
	\begin{proof}
		First, let $\Gamma$ be a graph with maximum degree at most $\lfloor n/2 \rfloor - 1$. Then the complement of $\Gamma$ has minimum degree at least $\lceil n/2 \rceil $, and by Theorem~\ref{thm:dirac} we find that $\overline{\Gamma}$ has a Hamiltonian cycle. It follows that the complement of this Hamiltonian cycle in $\overline{\Gamma}$ is a vertex-transitive closure of $\Gamma$, and so in this case $\dvt(\Gamma)\leq n-3$. Next, we will show that there exists a graph $\Gamma$ with maximum degree $\lfloor n/2 \rfloor$ and vertex-transitive number $n-1$. To simplify notation, let $k=\lfloor n/2 \rfloor$. 
		
		For odd $n$, let $n=2k+1$ and $\Gamma=K_{k,k} \oplus K_1$. Note that the complement of $\Gamma$ is $K_k \oplus K_{k}$ together with a single vertex $u$ adjacent to all the other vertices. Every vertex-transitive closure of $\Gamma$ corresponds to a vertex-transitive spanning subgraph $\Gamma'$ of $\overline{\Gamma}$, which is isomorphic to a disjoint union of copies of some connected vertex-transitive subgraph $H$ of $\overline {\Gamma}$.  Suppose that $\Gamma '$ contains a copy $H'$ of $H$ that contain vertices from both cliques of $\overline {\Gamma }$. Then $u$ is a cut-vertex of $H'$ and, since $H$ is vertex-transitive, every vertex of $H'$ must be a cut-vertex, which is impossible. Therefore each copy of $H$ in $\Gamma '$ intersects at most one of the cliques of $\overline{ \Gamma }$ and hence $|H|$ divides both $k$ and $k+1$, from which it follows that $|H| = 1$ and the only vertex-transitive closure of $\Gamma $ is $K_n$. 
		
		For even $n$, set $n=2k$ and $\Gamma = K_{k+1} \oplus K_{k-1}$. We will show that the only vertex-transitive spanning subgraph of $\overline{\Gamma}$ is the edgeless graph on $n$ vertices. First note that $\overline{\Gamma}=K_{k+1,k-1}$, and suppose to the contrary that $\overline{\Gamma}$ has a vertex-transitive spanning subgraph $\Gamma'$ with at least one edge. Since $\overline{\Gamma}$ has no perfect matching, by Theorem~\ref{thm:gr} we deduce that $\Gamma'$ must be disconnected. Let $H$ be a component of $\Gamma'$ and let $A$ and $B$ denote the two parts of the complete bipartite graph $K_{|A|,|B|}$ obtained as the subgraph of $\overline{\Gamma}$ induced by the vertex set of $H$. If $|H|$ is odd, then $H$ has a matching that misses exactly one vertex. By vertex-transitivity of $H$, for each vertex $v$ of $H$ there is a matching of $H$ that misses only $v$. However, there is no such matching if $v$ lies in the smaller of the sets $A$ and $B$. If $|H|$ is even, then $H$ has a perfect matching, and hence $|A|=|B|$. This must be true for each component $H$ of $\Gamma'$, but that is impossible, since the two parts of $\overline{\Gamma}$ do not have equal sizes. Hence we deduce that the only vertex-transitive spanning subgraph of $\overline{\Gamma}$ is the edgeless graph of order $n$, and conclude that $\dvt(\Gamma)=n-1$.
		
		Finally, note that if $\Gamma$ is a graph of order $n$ such that $\dvt(\Gamma)=n-1$, then every supergraph of $\Gamma$ on the same vertex set has vertex-transitive number $n-1$. Since the maximum degree of a supergraph (on the same vertex set) of $\Gamma$ can be any integer from $\Delta(\Gamma)$ to $n-1$, the rest follows easily.
	\end{proof}

	\section{Vertex-transitive numbers of graphs with maximum degree at most $2$}
	
	In this section, we investigate vertex-transitive numbers for graphs that have maximum degree at most $2$. If $\Gamma$ has maximum degree $0$, then it is clearly vertex-transitive, and we have $\dvt(\Gamma)=0$. The following observation addresses graphs with maximum degree $1$.
	
	\begin{proposition}\label{prop:maxdeg1}
		If $\Gamma$ is a graph of order $n$ with maximum degree $1$, then $\dvt(\Gamma)$ is equal to $1$ or $2$, depending on whether $n$ is even or odd.
	\end{proposition}
	\begin{proof}
		If $n$ is even, then $\Gamma$ is a disjoint union of at least one $K_2$ and an even number of isolated vertices. It follows that we can add edges to $E(\Gamma)$ to obtain a perfect matching on the $n$ vertices of $\Gamma$. This gives a vertex-transitive closure of $\Gamma$ of degree $1$, and hence $\dvt(\Gamma)=1$.
		
		If $n$ is odd, then by the Handshaking Lemma we find that the degree of every vertex-transitive closure of $\Gamma$ must be even, and hence $\dvt(\Gamma)\geq 2$. On the other hand, $\Gamma$ is clearly a spanning subgraph of the cycle on $n$ vertices, and it follows that $\dvt(\Gamma) = 2$.
	\end{proof}
	
	For the rest of this section, we let $\Gamma$ be a graph of order $n$ and maximum degree $2$. If $\Gamma$ is acyclic, then $\Gamma$ is a spanning subgraph of the cycle on $n$ vertices, and hence $d_{vt}(\Gamma)=2$. If $\Gamma$ contains a cycle, then its vertex-transitive number can be greater than $2$, but the following theorem shows that it is always bounded above by $4$.
	
	\begin{theorem}\label{thm:maxdeg2bound}
		The vertex-transitive number of a graph with maximum degree $2$ is at most $4$.
	\end{theorem}
	\begin{proof}
		Let $\Gamma$ be a graph of order $n$ and maximum degree $2$. We will show that $\Gamma$ is a spanning subgraph of the Cayley graph $\Gamma^* = \Cay(\ZZ_n,\{\pm 1,\pm 2\})$. Let $\Gamma'$ be a connected component of $\Gamma$ with $k$ vertices. Since $\Delta(\Gamma)=2$, we know that $\Gamma'$ is an isolated vertex, a path, or a cycle. We will show that in each case $\Gamma'$ is isomorphic to a subgraph of $\Gamma^*$ spanned by vertices $\{a,a+1,\dots,a+(k-1)\}$, where $a$ is any element of $\ZZ_n$. 
		
		This is clearly true if $\Gamma'$ is an isolated vertex. Also, since $\pm 1$ is in a generating set of $\Gamma^*$, the same holds in the case when $\Gamma'$ is a path. If $\Gamma'$ is a cycle and $k$ is even, then the paths $a,a+2,\dots,a+(k-2)$ and $a+1,a+3,\dots,a+(k-1)$ together with edges $\{a,a+1\}$ and $\{a+(k-2),a+(k-1)\}$ form a cycle of length $k$ in $\Gamma^*$. Finally, if $\Gamma'$ is a cycle and $k$ is odd, then a cycle of length $k$ in $\Gamma^*$ can be formed by the paths $a,a+2,\dots,a+(k-1)$ and $a+1,a+3,\dots,a+(k-2)$ and edges $\{a,a+1\}$ and $\{a+(k-2),a+(k-1)\}$. The proof now follows by induction on the number of connected components of $\Gamma$.
	\end{proof}
	
	The following observation characterises all graphs with maximum degree $2$ and vertex-transitive number $2$ (regardless of the parity of the order).
	
	\begin{proposition}\label{prop:dvt2}
		Let $\Gamma$ be a graph with maximum degree $2$. Then $\dvt(\Gamma)=2$ if and only if one of the following holds:
		\begin{enumerate}[label={\rm (\arabic*)},ref=(\arabic*)]
			\item\label{prop:dvt2:1} $\Gamma$ is acyclic;
			\item\label{prop:dvt2:2} all cycles in $\Gamma$ have the same length, call it $k$, and the multiset of all orders of the acyclic components of $\Gamma$ can be partitioned into subsets whose sums are $k$.
		\end{enumerate}
	\end{proposition} 
	\begin{proof}
		Every acyclic graph with maximum degree $2$ is the disjoint union of paths and isolated vertices, and hence $C_n$ is a vertex-transitive closure of every such graph. This proves (1).
		
		To prove (2), first note that every vertex-transitive graph with maximum degree $2$ must be the disjoint union of one or more cycles of the same length. Hence, if $\Gamma$ is not acyclic and $\dvt(\Gamma)=2$, then all cycles of $\Gamma$ must have the same length, say $k$. Moreover, the subgraph of $\Gamma$ induced by all acyclic components of $\Gamma$ must be a spanning subgraph of $C_k\oplus \dots \oplus C_k$. Noting that the acyclic components of $\Gamma$ are only paths and isolated vertices, it follows that if the disjoint union of any of these components has exactly $k$ vertices, then it is a spanning subgraph of $C_k$. Hence it is sufficient to partition the set of acyclic components into subsets $S$, such that the sum of the orders of all components in $S$ is $k$. On the other hand, if no such partition exists, then $\Gamma$ cannot be a spanning subgraph of a vertex transitive-graph of degree $2$, and the rest follows.  
	\end{proof}
	
	We now consider the case where the order $n$ of $\Gamma$ is odd. By the Handshaking Lemma we know that every vertex-transitive closure of $\Gamma$ must have even degree, and hence by Theorem~\ref{thm:maxdeg2bound} we have either $\dvt(\Gamma) = 2$ or $\dvt(\Gamma)=4$. Furthermore, since by Proposition~\ref{prop:dvt2} we can decide whether $\dvt(\Gamma) = 2$ or not, the problem of determining the vertex-transitive number for graphs of odd order and maximum degree $2$ reduces to the problem of deciding whether $\Gamma$ satisfies either of the properties \ref{prop:dvt2:1} and \ref{prop:dvt2:2} in Proposition~\ref{prop:dvt2}. While property \ref{prop:dvt2:1} can be easily checked in polynomial time, property  \ref{prop:dvt2:2} is equivalent to a well-known NP-complete problem usually referred to as multi-way number partitioning.
	
	So far we have shown that every graph with maximum degree $2$ has vertex-transitive number at most $4$, and we have characterised those with vertex-transitive number $2$. In order to complete the picture it is necessary to address the case when the vertex-transitive number is equal to $3$. This turns out to be difficult, but we will provide some partial observations. First of all, if some vertex-transitive closure of $\Gamma$ has degree $3$, then the Handshaking Lemma implies that the order of this vertex-transitive closure (and hence also the order of $\Gamma$) must be even. In the following proposition we consider the case when all cycles in $\Gamma$ have even length.
	
	\begin{proposition}\label{prop:noodd}
		Let $\Gamma$ be a graph of even order with maximum degree $2$. If there are no cycles of odd length in $\Gamma$, then $\dvt(\Gamma)\leq 3$.
	\end{proposition}
	\begin{proof}
		We will show that the prism graph $Y_n$ of degree $3$ is a vertex-transitive closure of $\Gamma$. Since $Y_n$ is Cayley (and hence vertex-transitive), this will prove the assertion. First, we transform $\Gamma$ into the graph $\Gamma'$ by connecting all acyclic components of $\Gamma$ into a single cycle. Note that the order of $\Gamma$ and the lengths of all of its cycles are even, and hence the length of the new cycle in $\Gamma'$ is even as well. It follows that $\Gamma'$ is a supergraph of $\Gamma$ on the same vertex set with no cycles of odd length and no acyclic components.
		
		Next, let $a_0,\dots,a_{n-1}$ denote the vertices in the outer cycle of $Y_n$, and $b_0,\dots,b_{n-1}$ the vertices of the inner cycle of $Y_n$, so that $a_i$ is adjacent to $b_i$ for each $i\in \{0,\dots,n-1\}$, and $a_i$ and $a_j$ (and, equivalently, $b_i$ and $b_j$) are adjacent if and only if $|i-j|$ is equal to $1$ or $n-1$. Now, it can be easily seen that for any $k,\ell\in \ZZ_n$ such that $\ell\geq 2$, the induced subgraph of $\Gamma'$ spanned by vertices $\{a_k,a_{k+1},\dots,a_{k+(\ell-1)},b_k,b_{k+1},\dots,b_{k+(\ell-1)}\}$ contains a Hamiltonian cycle of length $2\ell$. Hence, it follows by induction on the number of cycles in $\Gamma'$ that $Y_n$ is a vertex-transitive closure of $\Gamma'$, and consequently $Y_n$ is also a vertex-transitive closure of $\Gamma$.   
	\end{proof}
	
	The following example shows the necessity of the condition in Proposition~\ref{prop:noodd} for $|V(\Gamma)|$ to be even.
	
	\begin{example}
		Let $\Gamma=K_{2,2} \oplus K_1$. Note that $\Gamma$ has maximum degree $2$, and its only cycle has even length. On the other hand, by the proof of Theorem~\ref{thm:degrees} we know that the only vertex-transitive closure of $\Gamma$ is $K_5$, and hence $\dvt(\Gamma)=4$. Note that $\Gamma$ is the smallest example of a graph of odd order with no odd cycles that is not acyclic. (Recall that acyclic graphs with maximum degree $2$ have vertex-transitive number $2$.) It follows that $\Gamma$ is the smallest example (in terms of both $|V(\Gamma)|$ and $|E(\Gamma)|$) of a graph of odd order with maximum degree $2$ and no cycles of odd length.
	\end{example}
	
	In contrast with the case when all cycles of $\Gamma$ have even length, the problem of determining $\dvt(\Gamma)$ is far more complex even in the case when $\Gamma$ contains a single cycle of odd length and no other cycles. Let $|V(\Gamma)|=2k$ and assume that $\Gamma$ contains a single cycle (say of length $m$). If $m$ is even, then by Proposition~\ref{prop:noodd} we have $\dvt(\Gamma)\leq 3$, and hence the vertex-transitive number of $\Gamma$ can be determined by checking whether $\Gamma$ satisfies property~\ref{prop:dvt2:2} of Proposition~\ref{prop:dvt2}. If $m$ is odd and $m=k$, then $\Gamma$ is clearly a spanning subgraph of the disjoint union of two cycles of length $k$, and hence $\dvt(\Gamma)=2$. The case when $m>k$ is dealt with in the following theorem.
	
	\begin{theorem}\label{thm:odd}
		Let $\Gamma$ be a graph of order $2k$ with maximum degree $2$ and exactly one cycle. If the length $m$ of the cycle satisfies $k<m<2k$, then $\dvt(\Gamma)=3$.
	\end{theorem}
	\begin{proof}
		Since the sum of all orders of the acyclic components of $\Gamma$ is non-zero, but also smaller than $m$, by Proposition~\ref{prop:dvt2} we deduce that $\dvt(\Gamma)>2$. We proceed by considering two cases depending on whether $m-k$ is even or odd.
		
		In the case when $m-k$ is even, we will show that $\Gamma$ is a spanning subgraph of the prism graph $Y_k$. First, adopting the notation from the proof of Proposition~\ref{prop:noodd}, take the induced subgraph $A$ of $Y_k$ spanned by vertices $a_{m-k},\, a_{m-k+1},\, \dots,\, a_{k-1}$. Since $A$ is isomorphic to the path of order $2k-m$, which is exactly the number of vertices in all acyclic components of $\Gamma$, we deduce that the acyclic components of $\Gamma$ form a spanning subgraph of $A$. Next, let $B$ be the induced subgraph of $Y_k$ spanning all of its vertices not in $A$. Since $m-k$ is even, we find that edges $\{a_0,a_1\},\, \{a_2,a_3\},\, \dots,\, \{a_{m-k-2},a_{m-k-1}\}$, edges $\{b_1,b_2\},\, \{b_3,b_4\},\, \dots,\, \{b_{m-k-3},b_{m-k-2}\}$, and edges $\{a_i,b_i\}$ for $i\in \{0,1,\dots,m-k-1\}$ together with the path $(b_{m-k-1},b_{m-k},\dots,b_{k-1},b_0)$ form a Hamiltonian cycle in $B$ of length $m$, as required.
		
		In the case when $m-k$ is odd, we will show that $\Gamma$ is a spanning subgraph of the M\"{o}bius ladder $M_k$. We use the same notation for $M_k$ as for $Y_k$, with the only distinction that the edges $\{a_{k-1},a_0\}$ and $\{b_{k-1},b_0\}$ are swapped for edges $\{a_{k-1},b_0\}$ and $\{b_{k-1},a_0\}$. As in the previous case, acyclic components of $\Gamma$ form a spanning subgraph of the path $(a_{m-k}, a_{m-k+1}, \dots, a_{k-1})$ in $M_k$. An $m$-cycle in $M_k$ spanning the remaining vertices  can be formed by the path $(b_{m-k-1},b_{m-k},\dots, b_{k-1})$ together with the edge $\{b_{k-1},a_0\}$, edges $\{b_{2j-1},b_{2j}\}$ and $\{a_{2j},a_{2j+1}\}$ for $j\in \{0,1,\dots, (m-k-3)/2 \}$, and edges $\{a_i,b_i\}$ for $i\in \{0,1,\dots,m-k-1\}$.
	\end{proof}
	
	We already addressed the cases $m>k$ and $m=k$; next we look at the case $m<k$. Here we investigate only the simplest instance where $m=3$. Since $\Gamma$ contains a cycle, we have $\dvt(\Gamma) = 2$ if and only if $\Gamma$ satisfies property \ref{prop:dvt2:2} of Proposition~\ref{prop:dvt2}. If $\dvt(\Gamma) = 3$, then $\Gamma$ has a cubic vertex-transitive closure $\Gamma'$ of girth $3$. Since $\Gamma'$ is vertex-transitive, each of its components must be a connected vertex-transitive cubic graph of girth $3$. In~\cite{Eiben} it was shown that every such graph is isomorphic to $K_4$, $Y_3$, or a generalised truncation of an arc-transitive cubic graph by the triangle graph (the order of such a graph is always three times the order of the corresponding base graph). It follows that the order of $\Gamma'$ (and consequently also the order of $\Gamma$) must be divisible by $3$ or $4$. This, however, is not sufficient for $\Gamma$ to have vertex-transitive number $3$. Take for example  $K_3 \oplus P_{33}$ and suppose that it admits a cubic vertex-transitive completion $\Gamma'$. Clearly $\Gamma'$ must be connected, and so it is isomorphic to a generalised truncation of an arc-transitive cubic graph by the triangle graph. But this is not possible, since there are no arc-transitive cubic graphs of order $12$ (see~\cite{ConderDobcsanyi} for example), and so $\dvt(K_3 \oplus P_{33})=4$. As we can see, for $m=3$ the problem of determining $\dvt(\Gamma)$ is closely related to arc-transitive cubic graphs, knowledge of which is quite limited. We believe that further investigation both of this and the general case (where $m<k$) is warranted.

	\section{Graphs with large vertex-transitive numbers}
	
	In this section we consider graphs with large vertex-transitive numbers. Theorem~\ref{thm:degrees} implies that the gap between the vertex-transitive number and the maximum degree of a graph can be arbitrarily large. On the other hand, by Theorem~\ref{thm:maxdeg2bound} we know that if $\Gamma$ has maximum degree $2$, then $\dvt(\Gamma)\leq 4$, and hence in this case $\dvt(\Gamma)-\Delta(\Gamma)$ is at most $2$. The following theorem shows that there is no such bound for larger maximum degrees.
	
	\begin{theorem}\label{thm:unbounded}
		Let $S$ be the set of all graphs with maximum degree $d$. Then the set $\{\dvt(\Gamma) \mid \Gamma \in S \}$ is unbounded if and only if $d\geq 3$.     
	\end{theorem}
	\begin{proof}
		The assertion is clearly true for $d=0$, and for $d=1$ and $d=2$ it is an immediate consequence of Proposition~\ref{prop:maxdeg1} and Theorem~\ref{thm:maxdeg2bound}. Hence assume $d\geq 3$, and suppose to the contrary that there exists some $D$ such that the vertex-transitive number of every graph with maximum degree $d$ is at most $D$.
		
		Let $k$ be any positive integer, let $T$ be a $d$-ary tree of depth $k$ such that the root vertex has $d$ children and all vertices at depth smaller than $k$ have exactly $d-1$ children, and let $t$ denote the order of $T$. Note that $t=d^k+O(d^{k-1})$ as $k\to \infty$. Next, let $p$ be the largest prime less than or equal to $t$, and note that by Bertrand's postulate we have $t-p < t/2$. On the other hand, it can be easily seen that $t/2 < d^k$, and hence we can remove $t-p$ vertices from $T$ so that the resulting graph $G$ is a $d$-ary tree of depth $k$ and order greater than $t/2$. Note here that the order of $G$ is $p$, and so any vertex-transitive closure of $G$ must be a vertex-transitive graph of prime order. Since every vertex-transitive graph of prime order is a circulant (see~\cite{Turner} for example), it follows that every vertex-transitive closure of $G$ is a circulant.  
		
		Let $d'=\dvt(G)$, and let $G'$ be a vertex-transitive closure of $G$ of degree $d'$. Since $G$ has maximum degree $d$, we know that $d'\leq D$. The root vertex $v$ of $G$ is at distance at most $k$ from each vertex of $G$, and so the same is true for the corresponding vertex $v'$ of $G'$. Then, by vertex-transitivity of $G'$, we deduce that the diameter of $G$ is at most $k$. The order of a graph of degree $d'$ and diameter $k$ is bounded above by the Moore bound, which has asymptotic form $(d')^k+O\left((d')^{k-1}\right)$ as $k \to \infty$. In our case, however, we know that $G'$ is a Cayley graph of an abelian group, and in this situation we have a much better upper bound $M_{AC}(d',k)$ on $|V(G')|$ (see~\cite{Lewis} for details). In particular, we have the following asymptotic form:
		\[
		M_{AC}(d',k)=\begin{cases}
			(2^f/f!)k^f+O(k^{f-1}) & \text{for even }d',\text{ where } f=d'/2,\\
			(2^{f+1}/f!)k^f+O(k^{f-1}) & \text{for odd }d',\text{ where } f=(d'-1)/2.
		\end{cases}
		\]
		On the other hand, the order of $G'$ is equal to the order of $G$, and so it is greater than $t/2$. This gives a lower bound of the form $(d^k)/2+O(d^{k-1})$ as $k\to \infty$. It follows that the orders of graphs $G'$ (which exist for each positive integer $k$) can be bounded below by an exponential function in $k$, and bounded above by a polynomial function in $k$, a contradiction.\end{proof}
	
	Although the construction in the proof of Theorem~\ref{thm:unbounded} produces a graph for which the gap between its maximum degree and vertex-transitive number can be arbitrarily large (unless the maximum degree is smaller than or equal to $2$), this gap is negligibly small in comparison with the order of such a graph. This leads to the interesting problem of finding (or at least bounding) the largest possible gap between the maximum degree and the vertex-transitive number of a graph of order $n$. Denote this largest possible gap by $\gvt(n)$, and note that by Theorem~\ref{thm:degrees} for every $n\geq 4$ there exists a graph of order $n$ and maximum degree $\lfloor n/2 \rfloor$ with vertex-transitive number $n-1$. It follows that $\gvt(n)\geq n-1-\lfloor n/2 \rfloor$ for each $n\geq 4$. (It can be easily verified that this bound also holds for $n\in \{1,2,3\}$.) Remarkably, we have not found an example of a finite graph for which the gap between $\Delta(\Gamma)$ and $\dvt(\Gamma)$ is larger than $n-1-\lfloor n/2 \rfloor$. This creates an interesting open question.

	\section{Vertex-transitive numbers for locally finite graphs}
	
	The concept of a vertex-transitive number can be naturally extended to locally finite graphs. Let $\Gamma$ be a locally finite graph. If $\Gamma$ admits no locally finite vertex-transitive closure, we define $\dvt(\Gamma)=\infty$. Otherwise, $\dvt(\Gamma)$ is defined as the smallest integer $k$ for which there exists a $k$-regular locally finite vertex-transitive closure of $\Gamma$. 
	
	For finite graphs, Theorem~\ref{thm:maxdeg2bound} states that if a graph has maximum degree $2$, then its vertex-transitive number is at most $4$. In this case, every graph on $n$ vertices with maximum degree $2$ is a spanning subgraph of the Cayley graph $\Cay(\ZZ_n,\{\pm 1,\pm 2\})$; however, this does not hold for locally finite graphs. Here are two examples of locally finite graphs with maximum degree $2$ that are not spanning subgraphs of $\Cay(\ZZ,\{\pm 1,\pm 2\})$ (nor of disjoint unions of any number of such graphs): the one-way infinite path $P_{\infty}$ (Figure~\ref{fig:infinite:a}) and the disjoint union of the two-way infinite path $P^*_{\infty}$ with $C_3$ (Figure~\ref{fig:infinite:b}). As shown in Figure~\ref{fig:infinite}, both graphs admit locally finite vertex-transitive closures of degree $4$, but their structures are more complex. To date, we have not found any counterexample to the extension of Theorem~\ref{thm:maxdeg2bound} to locally finite graphs, and we leave this as an open question.
	
	\begin{figure}[htb]
\centering
\subfloat[Graph $P_{\infty}$\label{fig:infinite:a}]{
\begin{tikzpicture}[thick]
\draw[black!20] (-3,2.5) -- (3,2.5);
\draw[black!20] (-3,1.5) -- (3,1.5);
\draw[black!20] (-3,0.5) -- (3,0.5);
\draw[black!20] (-3,-0.5) -- (3,-0.5);
\draw[black!20] (-3,-1.5) -- (3,-1.5);
\draw[black!20] (-3,-2.5) -- (3,-2.5);

\draw[black!20] (2.5,-3) -- (2.5,3);
\draw[black!20] (1.5,-3) -- (1.5,3);
\draw[black!20] (0.5,-3) -- (0.5,3);
\draw[black!20] (-0.5,-3) -- (-0.5,3);
\draw[black!20] (-1.5,-3) -- (-1.5,3);
\draw[black!20] (-2.5,-3) -- (-2.5,3);

\draw[very thick] (-0.5,-0.5) -- (0.5,-0.5) -- (0.5,0.5) --  (-1.5,0.5)  -- (-1.5,-1.5) -- (1.5,-1.5) -- (1.5,1.5) -- (-2.5,1.5) -- (-2.5,-2.5) -- (2.5,-2.5) -- (2.5,2.5) -- (-2.5,2.5);
\draw[dotted, very thick] (-2.5,2.5) -- (-3,2.5);

\foreach \i in {1,2,...,6}
\foreach \j in {1,2,...,6}
\node[circle,draw,fill=white] at (-3.5+\i,-3.5+\j) {};

\end{tikzpicture}
}
\quad\quad\quad
\subfloat[Graph $P^*_{\infty} \oplus C_3$\label{fig:infinite:b}]{
\begin{tikzpicture}[thick]
\foreach \i in {1,2}
\coordinate (A\i) at (0:\i*2);
\foreach \i in {1,2,...,3}
\coordinate (B\i) at (60:\i);
\foreach \i in {1,2,...,3}
\coordinate (C\i) at (120:\i);
\foreach \i in {1,2}
\coordinate (D\i) at (180:\i*2);
\foreach \i in {1,2,...,3}
\coordinate (E\i) at (240:\i);
\foreach \i in {1,2,...,3}
\coordinate (F\i) at (300:\i);

\coordinate (X) at (0,0);

\foreach \i in {1,2,...,3}
\coordinate (P\i) at ($(D2)+(60:\i)$);

\foreach \i in {1,2,...,3}
\coordinate (Q\i) at ($(D2)+(300:\i)$);

\foreach \i in {1,2,...,3}
\coordinate (R\i) at ($(A2)+(120:\i)$);

\foreach \i in {1,2,...,3}
\coordinate (S\i) at ($(A2)+(240:\i)$);

\foreach \i in {1,2}
\coordinate (I\i) at ($(P1)+(0:\i)$);

\foreach \i in {1,2}
\coordinate (J\i) at ($(Q1)+(0:\i)$);

\foreach \i in {1,2}
\coordinate (K\i) at ($(C3)+(0:\i)$);

\foreach \i in {1,2}
\coordinate (L\i) at ($(E3)+(0:\i)$);

\foreach \i in {1,2}
\coordinate (M\i) at ($(B1)+(0:\i)$);

\foreach \i in {1,2}
\coordinate (N\i) at ($(F1)+(0:\i)$);

\draw[black!20] ($(P3)+(180:0.5)$) -- ($(R3)+(0:0.5)$);
\draw[black!20] ($(R3)+(120:0.5)$) -- ($(A2)+(300:0.5)$);
\draw[black!20] ($(A2)+(60:0.5)$) -- ($(S3)+(240:0.5)$);
\draw[black!20] ($(S3)+(0:0.5)$) -- ($(Q3)+(180:0.5)$);;
\draw[black!20] ($(Q3)+(300:0.5)$) -- ($(D2)+(120:0.5)$);;
\draw[black!20] ($(D2)+(240:0.5)$) -- ($(P3)+(60:0.5)$);;
\draw[black!20] ($(P1)+(180:0.5)$) -- ($(R1)+(0:0.5)$);
\draw[black!20] ($(C3)+(120:0.5)$) -- ($(F3)+(300:0.5)$);
\draw[black!20] ($(K2)+(120:0.5)$) -- ($(S2)+(300:0.5)$);
\draw[black!20] ($(P2)+(120:0.5)$) -- ($(L1)+(300:0.5)$);
\draw[black!20] ($(Q1)+(180:0.5)$) -- ($(S1)+(0:0.5)$);
\draw[black!20] ($(K1)+(60:0.5)$) -- ($(Q2)+(240:0.5)$);
\draw[black!20] ($(B3)+(60:0.5)$) -- ($(E3)+(240:0.5)$);
\draw[black!20] ($(R2)+(60:0.5)$) -- ($(L2)+(240:0.5)$);

\draw[very thick] (X) -- (E1) -- (F1) -- (X);
\draw[very thick] (C1) -- (B1) -- (M1) -- (A1) -- (N1) -- (F2) -- (L2) -- (F3) -- (S3) -- (S2) -- (N2) -- (S1) -- (A2) -- (R1) -- (M2) -- (R2) -- (R3);
\draw[very thick] (C1) -- (I2) -- (D1) -- (J2) -- (E2) -- (L1) -- (E3) -- (Q3) -- (Q2) -- (J1) -- (Q1) -- (D2) -- (P1) -- (I1) -- (P2) -- (P3) -- (C3) -- (C2) -- (K1) -- (K2) -- (B2) -- (B3);
\draw[dotted, very thick] (B3) -- ($(B3)+(60:0.5)$);
\draw[dotted, very thick] (R3) -- ($(R3)+(0:0.5)$);

\foreach \i in {1,2}
\node[circle,draw,fill=white] (A\i) at (0:\i*2) {};
\foreach \i in {1,2,...,3}
\node[circle,draw,fill=white] (B\i) at (60:\i) {};
\foreach \i in {1,2,...,3}
\node[circle,draw,fill=white] (C\i) at (120:\i) {};
\foreach \i in {1,2}
\node[circle,draw,fill=white] (D\i) at (180:\i*2) {};
\foreach \i in {1,2,...,3}
\node[circle,draw,fill=white] (E\i) at (240:\i) {};
\foreach \i in {1,2,...,3}
\node[circle,draw,fill=white] (F\i) at (300:\i) {};

\node[circle,draw,fill=white] (X) at (0,0) {};

\foreach \i in {1,2,...,3}
\node[circle,draw,fill=white] (P\i) at ($(D2)+(60:\i)$) {};

\foreach \i in {1,2,...,3}
\node[circle,draw,fill=white] (Q\i) at ($(D2)+(300:\i)$) {};

\foreach \i in {1,2,...,3}
\node[circle,draw,fill=white] (R\i) at ($(A2)+(120:\i)$) {};

\foreach \i in {1,2,...,3}
\node[circle,draw,fill=white] (S\i) at ($(A2)+(240:\i)$) {};

\foreach \i in {1,2}
\node[circle,draw,fill=white] (I\i) at ($(P1)+(0:\i)$) {};

\foreach \i in {1,2}
\node[circle,draw,fill=white] (J\i) at ($(Q1)+(0:\i)$) {};

\foreach \i in {1,2}
\node[circle,draw,fill=white] (K\i) at ($(C3)+(0:\i)$) {};

\foreach \i in {1,2}
\node[circle,draw,fill=white] (L\i) at ($(E3)+(0:\i)$) {};

\foreach \i in {1,2}
\node[circle,draw,fill=white] (M\i) at ($(B1)+(0:\i)$) {};

\foreach \i in {1,2}
\node[circle,draw,fill=white] (N\i) at ($(F1)+(0:\i)$) {};

\end{tikzpicture}
}
\caption{Vertex-transitive completions of degree $4$ for $P_{\infty}$ and $P^*_{\infty} \oplus C_3$}
\label{fig:infinite}
\end{figure}
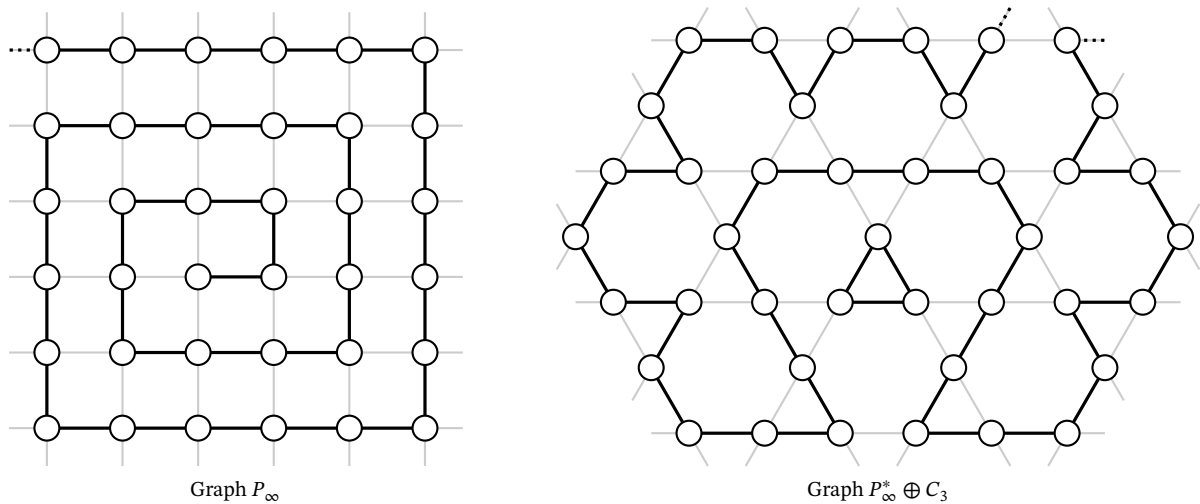

	We conclude this section with another open problem worth further investigation. A natural question in this setting is whether all locally finite graphs have a finite vertex-transitive number. This is clearly not true, as the vertex-transitive number of any locally finite graph with unbounded degrees must be infinite. The question becomes much more interesting when we restrict attention to locally finite graphs with bounded degrees. In this case, the problem can be stated as follows: Does there exist a locally finite graph with bounded maximum degree that admits no locally finite vertex-transitive supergraph on the same vertex set?


\bmsubsection*{Acknowledgments}
Martin Bachrat{\' y} acknowledges funding from the EU NextGenerationEU through the Recovery and Resilience Plan for Slovakia under the project No. 09I03-03-V04-00272. {\v S}tef{\' a}nia Glevitzk{\' a} acknowledges funding from the VEGA Research Grant 1/0437/23, and the Comenius University Grant UK/1436/2026. Jozef {\v S}ir{\' a}{\v n} acknowledges funding from the APVV Research Grants 22-0005 and 23-0076, and the VEGA Research Grants 1/0069/23 and 1/0011/25. James Tuite acknowledges funding from an LMS Early Career Fellowship (grant ECF-2021-27).

\bmsubsection*{Conflicts of Interest}
The authors declare no conflicts of interest.

\bibliography{vtclosures}


\end{document}